\documentclass[final,reqno]{siamltex}
\usepackage{latexsym,amsmath,amssymb,amsfonts,mathrsfs}
\usepackage{epsf,graphicx,epsfig,color,cite,cases}
\usepackage{subfigure,graphics,multirow,marginnote,enumerate,bm}
\newcommand{\Grad}{{\rm Grad\,}}

\newcommand{\Om}{\Omega}

\newcommand{\pa}{\partial}

\newcommand{\ov}{\overline}

\newcommand{\wid}{\widetilde}

\newcommand{\na}{\nabla}
\newcommand{\mat}{\mathbb}

\newcommand{\R}{{\mat R}}

\newcommand{\N}{{\mat N}}
\newcommand{\C}{{\mat C}}

\newcommand{\be}{\begin{eqnarray}}
\newcommand{\ben}{\begin{eqnarray*}}
\newcommand{\en}{\end{eqnarray}}
\newcommand{\enn}{\end{eqnarray*}}

\newtheorem{remark}[theorem]{Remark}

\begin{document}
\renewcommand{\theequation}{\arabic{section}.\arabic{equation}}

\title{\bf
Boundary determination for the Schr\"{o}dinger equation with unknown embedded obstacles by local data}
\author{Chengyu Wu\thanks{School of Mathematics and Statistics, Xi'an Jiaotong University,
Xi'an 710049, Shaanxi, China ({\tt wucy99@stu.xjtu.edu.cn})}
\and
Jiaqing Yang\thanks{School of Mathematics and Statistics, Xi'an Jiaotong University,
Xi'an 710049, Shaanxi, China ({\tt jiaq.yang@xjtu.edu.cn})}
}
\date{}
\maketitle


\begin{abstract}
  In this paper, we consider the inverse boundary value problem of the elliptic operator $\Delta+q$ in a fixed region $\Om\subset\R^3$ with unknown embedded obstacles $D$. In particular, we give a new and simple proof to uniquely determine $q$ and all of its derivatives at the boundary from the knowledge of the local Dirichlet-to-Neumann map on $\pa\Om$, disregarding the unknown obstacle, where in fact only the local Cauchy data of the fundamental solution is used. 
  Our proof mainly depends on the rigorous singularity analysis on certain singular solutions and the volume potentials of fundamental solution, which is easy to extend to many other cases. 
\end{abstract}



\pagestyle{myheadings}
\thispagestyle{plain}
\markboth{C. Wu and J. Yang}{Boundary determination for the Schr\"{o}dinger equation}

\section{Introduction}\label{sec1}
\setcounter{equation}{0}
Let $\Om\subset\R^3$ be a bounded domain with smooth boundary. Denote by $D\subset\Om$ an open subset such that $\pa D\in C^\infty$, $D\subset\subset\Om$ and $\Om\setminus\ov D$ is connected. In this paper, we are concerned with the inverse boundary value problem of the following system 
\be\label{1.1}
\left\{
\begin{array}{ll}
	\Delta u+qu=0~~~&{\rm in}~\Om\setminus\ov D,\\
	u=f~~~&{\rm on}~\pa\Om, \\
	\mathcal{B}(u)=0~~~&{\rm on}~\pa D,
\end{array}
\right.
\en
where the boundary data $f\in H^{1/2}(\pa\Om)$, $q\in C^\infty(\ov\Om,\C)$ and $\mathcal{B}$ stands for the boundary condition imposed on $\pa D$. To be exact, we shall consider the Dirichlet boundary condition $\mathcal{B}(u)=u$ or the Robin boundary condition $\mathcal{B}(u)=\pa_\nu u+\gamma u$ with $\nu$ the unit exterior normal on $\pa D$ and $\gamma\in C^\infty(\pa D,\C)$. Clearly, when $\gamma=0$ the Robin boundary condition is reduced to the Neumann boundary condition. 

Throughout the paper, we always assume that $0$ is not a Dirichlet eigenvalue of $\Delta+q$ in $\Om\setminus\ov D$ with the boundary condition $\mathcal{B}$ on $\pa D$. Under this assumption, it is known that problem (\ref{1.1}) is well-posed in $H^1(\Om\setminus\ov D)$. Therefore, the so-called Dirichlet-to-Neumann (DtN) map $\Lambda:f\mapsto\pa_\nu u|_{\pa\Om}$ is well-defined and acts as a bounded linear operator from $H^{1/2}(\pa\Om)$ into $H^{-1/2}(\pa\Om)$. In this paper, we will study the inverse problem of recovering $(q,D,\mathcal{B})$ from the local knowledge of $\Lambda$. In particular, let $\Gamma\subset\pa\Om$ be a nonempty open subset. We shall prove: 
\begin{theorem}\label{thm1.1}
	For $(q_i,D_i,\mathcal{B}_i)$, $i=1,2$, suppose $\Lambda_1f=\Lambda_2f$ on $\Gamma$ for all $f\in H^{1/2}(\pa\Om)$ with ${\rm{supp}}f\subset\Gamma$,  
	then 
	\ben
	 D^\alpha q_1=D^\alpha q_2~on~\Gamma~{\text{for}}~all~|\alpha|\geq0.
	\enn
\end{theorem}Moreover, if $q$ is analytic, we obtain the following global uniqueness. 
\begin{theorem}\label{thm1.2}
	For $(q_i,D_i,\mathcal{B}_i)$, $i=1,2$, suppose $\Lambda_1f=\Lambda_2f$ on $\Gamma$ for all $f\in H^{1/2}(\pa\Om)$ with ${\rm{supp}}f\subset\Gamma$, and further that $q_i$ is analytic in $\ov\Om\setminus D_i$, $i=1,2$, then $D_1=D_2=D$, $q_1=q_2$ in $\Om\setminus\ov D$ and $\mathcal{B}_1=\mathcal{B}_2$. 
\end{theorem}

The inverse boundary value problem was initiated by Calder\'{o}n \cite{AP80}, where he posed the question of recovering the conductivity of a medium by making voltage and current measurements at the boundary. Since then, substantial progress has been made on this problem in the case where the measurements are on the whole boundary. We refer the readers to some fundamental papers like \cite{GA90,RM84,AI95,JG87,JG88}. Recently, the partial data case has arisen more intersets with relating work such as \cite{AG02,OGM10,CJG07,VI07}. Nevertheless, due to the great difficulty brought by the incompleteness of the information of the DtN map, it still remains widely open for the general partial data problem and we refer to the survey  \cite{CM14} for more details. 

The global uniqueness usually along with the boundary determination results, such as \cite{GA90,RM84,RM85,JG88,JG89}. Kohn and Vogelius \cite{RM84,RM85} uniquely determined the conductivity and all of its derivatives at the boundary using the highly oscillatory solutions, while Uhlmann and his collaborators \cite{JG88,JG89} proved the similar results applying the methods of microlocal analysis by computing the full symbol of the DtN map. Alessandrini \cite{GA90} constructed solutions with isolated singularity of aribitrary order to obtain the uniqueness and stability results for the boundary determination of Lipschitz domains and coefficients. 

Compared with the previous work, we consider the case with somewhat more generality, i.e., $(D,\mathcal{B})$ is unknown. For the inverse problem with the presence of the obstacle D, we refer some relevant work \cite{VI08,VI88,JMK21}. In this paper, we propose a novel and primary method to give a new proof and slightly improve the uniqueness result on the boundary determination of the potential $q$ when $(D,\mathcal{B})$ is unknown. Particularly, we use only the local Cauchy data of the fundamental solution to show that the potential $q$ is uniquely determined to infinite order at the boundary, disregarding the unknown obstacle. Our methods mainly depends on the detailed singularity analysis for certain singular solutions and some other relating functions. Specifically, we utilize the Green's representatioin to derive the full singularity of the solutions, and rigorously study the derivatives of $1/\sqrt{s^2+t^2}$ and some associated singular integrals to obtain the singularity of the volume potential of the fundamental solution. These explicit singularity allows us to ignore the unknown obstacle and thus we can directly compare the singularity of solutions corresponding to different potentials $q$ locally at the boundary. Then any diffenence of $q$ at the boundary would contradict with the singularity obtained previously and hence the boundary determination result follows. Furthermore, when the potential $q$ is known to be analytic, we use the Green's functions to deduce the global uniqueness, i.e., the unique determination of $(q,D,\mathcal{B})$. It is noted that we are only using the solutions with the Dirichlet data simply being $1/|x-z|$, the fundamental solution of the principal term of the concerning differential operator $\Delta+q$, while in \cite{RM84,RM85} some cleverly chosen sequences of Dirichlet data are utilized. Further, in \cite{GA90} Alessandrini constructed solutions behaving like $1/|x-z|^{m}$ for aribitrary $m$, while we are deriving the singularity in the $H^m$ sense for aribitrary $m$ for fixed solutions behaving like $1/|x-z|$, which is more handleable when extends to other equations. In fact, in a forthcoming paper we will show that our methods can be applied to derive the boundary determination result for the Maxwell equation using local data. In summary, the present study propose a rather systematic way for the boundary determination of the linear differential operator whose principal term has known fundamental solution, which is a fairly general situation. 

The paper is organized as follows. In section \ref{sec2}, we do some elaborate singularity analysis to prepare for the proof of the main theorems. In particular, in section \ref{sec2.1} we propose a simple method to derive the complete singularity of the solutions to problem \eqref{1.1} with certain singular boundary data. Then in section \ref{sec2.2}, we compute explicitly the norms of the volume potential of the fundamental solution. Finally, in section \ref{sec3} we prove the main results Theorems \ref{thm1.1} and \ref{thm1.2}.

\section{Some preparations}\label{sec2}
\setcounter{equation}{0}
In this section, we show some technical results about singularity for the preparations of the main theorems. We shall employ a novel method to derive the complete description of the singularity of the solution to problem (\ref{1.1}) with the boundary data being the fundamental solution to the Laplacian. The method mainly depends on the Green's representation of the solutions and thus can be applied to many other cases. Then we compute the Sobolev norms at the boundary for the volume potential of fundamental solution. In particular, we show that these norms would explode, which plays a significant role in later proof. 

\subsection{Singularity of solutions}\label{sec2.1}
We first introduce some useful notations. Let $B_r(x)$ denote the open ball centered at $x\in\R^3$ with radius $r>0$. For balls centered at the origin, we abbreviate by $B_r$. Denote by $\Phi(x,y)$ the fundamental solution to the Laplacian in $\R^3$, which is 
\ben
  \Phi(x,y)=\frac{1}{4\pi|x-y|},~~~\quad x\neq y. 
\enn
Define the single- and double-layer potentials by 
\ben
(\mathcal{S}\varphi)(x)&&:=\int_{\pa\Om}\Phi(x,y)\varphi(y)ds(y),~~x\in\R^3\setminus\pa\Om, \\
(\mathcal{D}\varphi)(x)&&:=\int_{\pa\Om}\frac{\pa\Phi(x,y)}{\pa\nu(y)}\varphi(y)ds(y),~~x\in\R^3\setminus\pa\Om. 
\enn
Also, we give the definitions of the associated boundary integral operators, for $x\in\pa\Om$, 
\ben
(S\varphi)(x)&&:=\int_{\pa\Om}\Phi(x,y)\varphi(y)ds(y), \\
(K\varphi)(x)&&:=\int_{\pa\Om}\frac{\pa\Phi(x,y)}{\pa\nu(y)}\varphi(y)ds(y), \\
(K'\varphi)(x)&&:=\int_{\pa\Om}\frac{\pa\Phi(x,y)}{\pa\nu(x)}\varphi(y)ds(y),\\
(T\varphi)(x)&&:=\frac{\pa}{\pa\nu(x)}\int_{\pa\Om}\frac{\pa\Phi(x,y)}{\pa\nu(y)}\varphi(y)ds(y). 
\enn
Further, the volume potential is given by 
\ben
(\mathcal{G}_{q}\varphi)(x)&&:=\int_{\Om}\Phi(x,y)q(y)\varphi(y)dy,~~~x\in\R^3. 
\enn
The restriction of the volume potential $\mathcal{G}_{q}\varphi$ on the boundary $\pa\Om$ and the corresponding normal derivative on $\pa\Om$ are denoted by $G_{q}\varphi$ and $\pa_\nu G_{q}\varphi$, respectively. Since we require $\pa\Om\in C^\infty$ and $q\in C^\infty(\ov\Om,\C)$, the following mapping properties hold (see details in \cite{GW08,WM00}). 
\begin{theorem}\label{thm2.1}
	For $s\geq0$, the potentials $\mathcal{S}:H^{s-1/2}(\pa\Om)\rightarrow H^{s+1}(\Om)$, $\mathcal{D}:H^{s+1/2}(\pa\Om)\rightarrow H^{s+1}(\Om)$ and $\mathcal{G}_{q}:H^s(\Om)\rightarrow H^{s+2}(\Om)$ are bounded. Moreover, the boundary integral operators $S,K,K':H^{s-1/2}(\pa\Om)\rightarrow H^{s+1/2}(\pa\Om)$ and $T:H^{s+1/2}(\pa\Om)\rightarrow H^{s-1/2}(\pa\Om)$ are bounded for $s\geq0$. 
\end{theorem}
\begin{theorem}\label{thm2.2}
	For $x_0\in\pa\Om$ and $\delta>0$ small enough, define $z_j:=x_0+(\delta/j)\nu(x_0)\in\R^3\setminus\ov\Om$. Let $u_j$ be the unique solution to problem {\rm (\ref{1.1})} with the boundary data $f_j=\Phi(x,z_j)$. Set $\varphi_{-1,j}=\varphi_{0,j}=0$ and $\psi_{-1,j}=\psi_{0,j}=0$. For $m\in\N\cup\{0\}$, define 
	\ben
	  &&\varphi_{m+1,j}:=2K'\varphi_{m,j}+2\pa_\nu G_{q}(\Phi(\cdot,z_j)+\psi_{m-1,j})~~\qquad{\rm on}~\pa\Om, \\
	  &&\psi_{m+1,j}:=\mathcal{S}\varphi_{m+1,j}+\mathcal{G}_{q}(\Phi(\cdot,z_j)+\psi_{m-1,j})\qquad\quad\;\;\;~~{\rm in}~\ov\Om. 
	\enn
	Then we have the following estimates 
	\ben
	  &&\|\pa_\nu(u_j-\Phi(\cdot,z_j))-\varphi_{m,j}\|_{H^{m+1/2}(\pa\Om)}\leq C, \\
	  &&\|u_j-\Phi(\cdot,z_j)-\psi_{m,j}\|_{H^{m+2}(\Om\setminus\ov D)}\leq C 
	\enn
	for $m\in\N\cup\{-1,0\}$, where $C>0$ is a constant independent of $j\in\N$. 
\end{theorem}
\begin{proof}
	Without loss of genearlity, we assume that $0$ is not a Dirichlet eigenvalue of $\Delta+q$ in $\Om$. Then $\wid u_j$ given by 
	\be\label{2.1}
	\left\{
	\begin{array}{ll}
		\Delta\wid u_j+q\wid u_j=0~~~&{\rm in}~\Om,\\
		\wid u_j=\Phi(\cdot,z_j)~~~&{\rm on}~\pa\Om 
	\end{array}
	\right.
	\en
	is well-defined. Further, it can be verified that $U_j:=u_j-\wid u_j$ in $\Om\setminus\ov D$ satisfies  
	\be\label{2.2}
	\left\{
	\begin{array}{ll}
		\Delta U_j+qU_j=0~~~&{\rm in}~\Om\setminus\ov D,\\
		U_j=0~~~&{\rm on}~\pa\Om, \\ 
		\mathcal{B}(U_j)=-\mathcal{B}(\wid u_j)~~~&{\rm on}~\pa D. 
	\end{array}
	\right.
	\en
	
	We claim that for $m\in\N\cup\{-1,0\}$, 
	\be\label{2.3}
	&&\|\pa_\nu(\wid u_j-\Phi(\cdot,z_j))-\varphi_{m,j}\|_{H^{m+1/2}(\pa\Om)}\leq C, \\ \label{2.4}
	&&\|\wid u_j-\Phi(\cdot,z_j)-\psi_{m,j}\|_{H^{m+2}(\Om)}\leq C 
	\en
	with the constant $C>0$ independent of $j\in\N$. To begin with, we rewrite (\ref{2.1}) as 
	\be\label{2.5}
	\left\{
	\begin{array}{ll}
		(\Delta+q)(\wid u_j-\Phi(\cdot,z_j))=-q\Phi(\cdot,z_j)~~~&{\rm in}~\Om,\\
		\wid u_j-\Phi(\cdot,z_j)=0~~~&{\rm on}~\pa\Om. 
	\end{array}
	\right.
	\en
	Since $\Phi(\cdot,z_j)$ is uniformly bounded in $L^2(\Om)$ for $j\in\N$, we deduce from the standard elliptic regularity that $\|\wid u_j-\Phi(\cdot,z_j)\|_{H^2(\Om)}\leq C$ uniformly for $j\in\N$, which implies by the trace theorem that (\ref{2.3}) and (\ref{2.4}) hold for $m=-1,0$. Further, from (\ref{2.5}) we see the representation 
	\be\label{2.6}
	  \wid u_j(x)-\Phi(x,z_j)=[\mathcal{S}(\pa_\nu(\wid u_j-\Phi(\cdot,z_j)))](x)+(\mathcal{G}_{q}\wid u_j)(x),~~~x\in\Om. 
	\en
	Taking the normal derivative on $\pa\Om$ yields that 
	\be\label{2.7}
	\pa_\nu(\wid u_j(x)-\Phi(x,z_j))=2[K'(\pa_\nu(\wid u_j-\Phi(\cdot,z_j)))](x)+2(\pa_\nu G_{q}\wid u_j)(x) 
	\en
	for $x\in\pa\Om$. Now we assume that (\ref{2.3}) and (\ref{2.4}) hold for $m\leq l$ with $l\geq0$. Then it follows from (\ref{2.7}) that on $\pa\Om$ 
	\ben
	  \pa_\nu(\wid u_j-\Phi(\cdot,z_j))-\varphi_{l+1,j}=\;&&2K'(\pa_\nu(\wid u_j-\Phi(\cdot,z_j))-\varphi_{l,j}) \\
	  &&+2\pa_\nu G_{q}(\wid u_j-\Phi(\cdot,z_j)-\psi_{l-1,j}). 
	\enn
	Applying the estimates (\ref{2.3}) and (\ref{2.4}) when $m=l-1,l$, by Theorem \ref{thm2.1} and the trace theorem we derive that $\|\pa_\nu(\wid u_j-\Phi(\cdot,z_j))-\varphi_{l+1,j}\|_{H^{l+3/2}(\pa\Om)}\leq C$. Furthermore, from (\ref{2.6}) it can be verified that in $\Om$
	\ben
	  \wid u_j-\Phi(\cdot,z_j)-\psi_{l+1,j}=\;&&\mathcal{S}[\pa_\nu(\wid u_j-\Phi(\cdot,z_j))-\varphi_{l+1,j})] \\
	  &&+\mathcal{G}_{q}(\wid u_j-\Phi(\cdot,z_j)-\psi_{l-1,j}). 
	\enn
	Combining the inequalities (\ref{2.3}) with $m=l+1$ and (\ref{2.4}) with $m=l-1$, again by Theorem \ref{thm2.1} we deduce that $\|\wid u_j-\Phi(\cdot,z_j)-\psi_{l+1,j}\|_{H^{l+3}(\Om)}\leq C$ uniformly for $j\in\N$. Hence, we obtain that (\ref{2.3}) and (\ref{2.4}) hold for $m=l+1$. The induction arguments then yield that (\ref{2.3}) and (\ref{2.4}) hold for all $m\in\N\cup\{-1,0\}$. 
	
	Due to the positive distance between $z_j$ and $D$, from (\ref{2.5}) and the interior regularity of elliptic equations we see that  $\|\wid u_j\|_{H^{m+2}(D)}\leq C$ for all $m\in\N\cup\{-1,0\}$, which implies by the trace theorem that $\|\mathcal{B}(\wid u_j)\|_{H^s(\pa D)}\leq C$ for $s\geq0$. It then follows from (\ref{2.2}) and the global elliptic regularity that $\|U_j\|_{H^m(\Om\setminus\ov D)}\leq C$ for $m\in\N$. Therefore, $u_j=\wid u_j+U_j$ in $\Om\setminus\ov D$ satisfies the desired estimates and the proof is complete. 
\end{proof}
\begin{remark}\label{remark2.3}
	The assumption that $0$ is not a Dirichlet eigenvalue of $\Delta+q$ in $\Om$ can be removed by analyzing $u_j$ directly. Here we spilt $u_j$ into two parts for simplicity. By induction arguments, it is easy to see that for $m\in\N$, $\|\varphi_{m,j}\|_{H^{1/2}(\pa\Om)}$ and $\|\psi_{m,j}\|_{H^2(\Om)}$ are uniformly bounded for $j\in\N$. Further, by the trace theorem and the interior regularity, for $m\in\N$ and $\varepsilon>0$ small we have that $\|\varphi_{m,j}\|_{H^{m+1/2}(\pa\Om\setminus\overline{B_\varepsilon(x_0)})}+\|\psi_{m,j}\|_{H^{m+3/2}(\pa\Om)}\leq C$ uniformly for $j\in\N$. 
\end{remark}

\subsection{The exploding norms}\label{sec2.2}
In this subsection, we show that the Sobolev norms of the volume potential of fundamental solution at the boundary would explode by deriving the explicit expression of the high order derivatives for $1/\sqrt{s^2+t^2}$ and investigating on the relating singular integral, which may have many other applications and deserves more study. 
\begin{lemma}\label{lem2.3}
	For $m\in\N\cup\{0\}$, we have that 
	\be\label{2.8}
	  \pa_s^m\left(\frac{1}{\sqrt{s^2+t^2}}\right)=\frac{P_m(s,t)}{\left(\sqrt{s^2+t^2}\right)^{2m+1}}. 
	\en
	Here $P_m$ is a sequence of polynomial given by 
	\ben
	  P_{2m}(s,t)=\sum_{i=0}^{m}a_{2m,2i}s^{2i}t^{2(m-i)},~~~P_{2m+1}(s,t)=\sum_{i=0}^{m}a_{2m+1,2i+1}s^{2i+1} t^{2(m-i)}
	\enn
	with $m\in\N\cup\{0\}$, where 
    \ben
      a_{2m,2i}\;&&=(-1)^{m-i}2^{2i-2m}(2m)!C_{2m}^{m-i}C_{m+i}^{m-i}, \\
      a_{2m+1,2i+1}\;&&=(-1)^{m-i+1}2^{2i-2m}(2m+1)!C_{2m+1}^{m-i}C_{m+i+1}^{m-i} 
    \enn
	for $m\in\N\cup\{0\}$ and $0\leq i\leq m$. {\rm (}Here we set $C_0^0=1$.{\rm )} Moreover, the following identities hold 
	\be\label{2.9}
	  &&\sum_{i=0}^m\frac{a_{2m,2i}}{C_{2m-1}^{m-i}}=-\frac{2}{2m+1}\sum_{i=0}^m\frac{a_{2m+1,2i+1}}{C_{2m}^{m-i}}=\frac{(2m)!}{2^{2m-1}}, \\ \label{2.10}
	  &&(m+1)\sum_{i=0}^m\frac{a_{2m,2i}}{C_{2m}^{m-i}}=-\sum_{i=0}^m\frac{a_{2m+1,2i+1}}{C_{2m+1}^{m-i}}=\frac{(2m+2)!}{2^{2m+1}} 
	\en
	for $m\in\N$. 
\end{lemma}
\begin{proof}
	We first show \eqref{2.8}, which is clear when $m=0,1$. Now we assume \eqref{2.8} holds for $m\leq l$ with $l\geq1$. Calculating straightly yields that 
	\ben
	   \pa_s^{l+1}\left(\frac{1}{\sqrt{s^2+t^2}}\right)&&=\pa_s\left(\frac{P_l(s,t)}{\left(\sqrt{s^2+t^2}\right)^{2l+1}}\right) \\
	   &&=\frac{\left(\sqrt{s^2+t^2}\right)^{2l+1}\pa_sP_l(s,t)-(2l+1)s\left(\sqrt{s^2+t^2}\right)^{2l-1}P_l(s,t)}{\left(\sqrt{s^2+t^2}\right)^{4l+2}} \\
	   &&=\frac{(s^2+t^2)\pa_sP_l(s,t)-(2l+1)sP_l(s,t)}{\left(\sqrt{s^2+t^2}\right)^{2l+3}}, 
	\enn
	which implies the recurrence formula  
	\be\label{2.11}
	  P_{l+1}(s,t)=(s^2+t^2)\pa_sP_l(s,t)-(2l+1)sP_l(s,t). 
	\en
	If $l=2l_0$ is even, by the explicit form of $P_{2l_0}$ we see that 
	\ben
	  P_{2l_0+1}(s,t)\;&&=(s^2+t^2)\sum_{i=1}^{l_0}2ia_{2l_0,2i}s^{2i-1}t^{2(l_0-i)}-(4l_0+1)s\sum_{i=0}^{l_0}a_{2l_0,2i}s^{2i}t^{2(l_0-i)} \\
	  &&=\sum_{i=0}^{l_0-1}[(2i-4l_0-1)a_{2l_0,2i}+(2i+2)a_{2l_0,2i+2}]s^{2i+1}t^{2(l_0-i)} \\
	  &&\quad-(2l_0+1)a_{2l_0,2l_0}s^{2l_0+1}. 
	\enn
	It then follows from the basic calculation that for $0\leq i\leq l_0-1$ 
	\ben
	  &&\quad(2i-4l_0-1)a_{2l_0,2i}+(2i+2)a_{2l_0,2i+2} \\
	  &&=(-1)^{l_0-i+1}2^{2i-2l_0}(2l_0+1)!C_{2l_0+1}^{l_0-i}C_{l_0+i+1}^{l_0-i}=a_{2l_0+1,2i+1}, 
	\enn
	and $-(2l_0+1)a_{2l_0,2l_0}=-(2l_0+1)!=a_{2l_0+1,2l_0+1}$, and thus 
	\ben
	  P_{l+1}(s,t)=P_{2l_0+1}(s,t)=\sum_{i=0}^{l_0}a_{2l_0+1,2i+1}s^{2i+1} t^{2(l_0-i)}. 
	\enn
	Through the same way we can deal with the case that $l$ is odd. Therefore, the induction arguments show that \eqref{2.8} holds for $m\in\N\cup\{0\}$. 
	
	Next we verify the identities \eqref{2.9} and \eqref{2.10}. To this end, we derive a recurrence formula different from \eqref{2.11}. For $m\in\N$, it is clear to see that 
	\ben
	  \pa_s^m\left((s^2+t^2)\frac{1}{\sqrt{s^2+t^2}}\right)=\pa_s^{m-1}\left(\frac{s}{\sqrt{s^2+t^2}}\right). 
	\enn
	Applying Leibniz's formula and \eqref{2.8}, we deduce that for $m\geq2$ 
	\ben
	   &&\quad\pa_s^m\left((s^2+t^2)\frac{1}{\sqrt{s^2+t^2}}\right) \\
	   &&=\frac{P_m(s,t)+2msP_{m-1}(s,t)+(m-1)m(s^2+t^2)P_{m-2}(s,t)}{\left(\sqrt{s^2+t^2}\right)^{2m-1}}
	\enn
	and 
	\ben
	  \pa_s^{m-1}\left(\frac{s}{\sqrt{s^2+t^2}}\right)
	  =\frac{sP_{m-1}(s,t)+(m-1)(s^2+t^2)P_{m-2}(s,t)}{\left(\sqrt{s^2+t^2}\right)^{2m-1}}, 
	\enn
	which indicates that 
	\be\label{2.12}
	  P_m(s,t)+(2m-1)sP_{m-1}(s,t)+(m-1)^2(s^2+t^2)P_{m-2}(s,t)=0. 
	\en
	Combining \eqref{2.11} and \eqref{2.12}, we obtain the recurrence relation  
	\be\label{2.13}
	  \pa_sP_m(s,t)=-m^2P_{m-1}(s,t),~~~m\in\N. 
	\en
	Utilizing the expression of $P_{2m}$ and $P_{2m+1}$, it is derived from \eqref{2.11} and \eqref{2.13} that for $m\geq2$, 
	\be\label{2.14}
	a_{2m,2i}\;&&=
	\left\{
	\begin{array}{cc}
		a_{2m-1,1},~~~&i=0, \\ 
		(2i-4m)a_{2m-1,2i-1}+(2i+1)a_{2m-1,2i+1},~~~&1\leq i\leq m-1, \qquad\\ 
		-2ma_{2m-1,2m-1},~~~&i=m, 
	\end{array}
	\right. \\ \label{2.15}
	a_{2m+1,2i+1}\;&&=
	\left\{
	\begin{array}{cc}
		(2i-4m-1)a_{2m,2i}+(2i+2)a_{2m,2i+2},~~~&0\leq i\leq m-1, \\
		-(2m+1)a_{2m,2m},~~~&i=m. 
	\end{array}
	\right. 
	\en
	and for $m\in\N$, 
	\be\label{2.16}
	  a_{2m,2i}\;&&=-\frac{(2m)^2}{2i}a_{2m-1,2i-1},~~~1\leq i\leq m, \\ \label{2.17}
	  a_{2m+1,2i+1}\;&&=-\frac{(2m+1)^2}{2i+1}a_{2m,2i},~~~~0\leq i\leq m. 
	\en
	
	We shall use the recurrence relations \eqref{2.14}--\eqref{2.17} and the basic properties of combinatorial numbers to prove the identities \eqref{2.9} and \eqref{2.10}. It is first deduced by calculating directly that \eqref{2.9}, \eqref{2.10} hold when $m=1$. Now suppose that \eqref{2.9} and \eqref{2.10} are satisfied for $m\leq l$ with $l\geq1$. By \eqref{2.14} we have that 
	\be\label{2.18}\nonumber
	  \sum_{i=0}^{m+1}\frac{a_{2(m+1),2i}}{C_{2m+1}^{m+1-i}}\;&&=\sum_{i=0}^m\frac{(m+1-i)a_{2(m+1),2i}}{(2m+1)C_{2m}^{m-i}}+a_{2(m+1),2(m+1)} \\ \nonumber
	  &&=\sum_{i=1}^m\frac{(m+1-i)(2i-4m-4)a_{2m+1,2i-1}}{(2m+1)C_{2m}^{m-i}} \\ \nonumber
	  &&\quad+\sum_{i=0}^m\frac{(m+1-i)(2i+1)a_{2m+1,2i+1}}{(2m+1)C_{2m}^{m-i}}-2(m+1)a_{2m+1,2m+1} \\ \nonumber
	  &&=\sum_{i=0}^m[(m+i+1)-(2m+1)(2m+2)]\frac{a_{2m+1,2i+1}}{(2m+1)C_{2m}^{m-i}} \\ 
	  &&=-(2m+2)\sum_{i=0}^m\frac{a_{2m+1,2i+1}}{C_{2m}^{m-i}}+\sum_{i=0}^m\frac{a_{2m+1,2i+1}}{C_{2m+1}^{m-i}} 
	\en
	with $m\in\N$. Following the similar procedure, we derive from \eqref{2.15} that for $m\in\N$ 
	\be\label{2.19}
	  \sum_{i=0}^{m+1}\frac{a_{2m+3,2i+1}}{C_{2m+2}^{m+1-i}}=-(2m+3)\sum_{i=0}^{m+1}\frac{a_{2(m+1),2i}}{C_{2m+1}^{m+1-i}}+\sum_{i=0}^{m+1}\frac{a_{2(m+1),2i}}{C_{2m+2}^{m+1-i}}. 
	\en
	Further, in view of \eqref{2.16}, we deduce that 
	\be\label{2.20}\nonumber
	  \sum_{i=0}^m\frac{a_{2m+1,2i+1}}{C_{2m+1}^{m-i}}\;&&=\sum_{i=1}^{m+1}\frac{a_{2m+1,2i-1}}{C_{2m+1}^{m+1-i}}=-\sum_{i=0}^{m+1}\frac{ia_{2(m+1),2i}}{2(m+1)^2C_{2m+1}^{m+1-i}} \\ \nonumber
	  &&=-\sum_{i=0}^{m+1}\frac{(m+1+i)a_{2(m+1),2i}}{2(m+1)^2C_{2m+1}^{m+1-i}}+\sum_{i=0}^{m+1}\frac{a_{2(m+1),2i}}{(2m+2)C_{2m+1}^{m+1-i}} \\
	  &&=\frac{1}{2m+2}\sum_{i=0}^{m+1}\frac{a_{2(m+1),2i}}{C_{2m+1}^{m+1-i}}-\frac{1}{2m+2}\sum_{i=0}^{m+1}\frac{2a_{2(m+1),2i}}{C_{2m+2}^{m+1-i}}, 
	\en
	where $m\in\N$. Analyzing analogously, by \eqref{2.17}, we can obtain that 
	\be\label{2.21}
	  \sum_{i=0}^m\frac{a_{2m,2i}}{C_{2m}^{m-i}}=\frac{1}{2m+1}\sum_{i=0}^{m}\frac{a_{2m+1,2i+1}}{C_{2m}^{m-i}}-\frac{1}{2m+1}\sum_{i=0}^{m}\frac{2a_{2m+1,2i+1}}{C_{2m+1}^{m-i}} 
	\en
	for $m\in\N$. It is then seen from \eqref{2.18} and \eqref{2.20}, and \eqref{2.19} and \eqref{2.21} that for $m\in\N$ 
	\be\label{2.22}
	  &&\sum_{i=0}^{m+1}\frac{a_{2(m+1),2i}}{C_{2m+2}^{m+1-i}}=-(m+1)\sum_{i=0}^m\frac{a_{2m+1,2i+1}}{C_{2m}^{m-i}}-\frac{2m+1}{2}\sum_{i=0}^m\frac{a_{2m+1,2i+1}}{C_{2m+1}^{m-i}}, \\ \label{2.23}
	  &&\sum_{i=0}^{m+1}\frac{a_{2m+3,2i+1}}{C_{2m+3}^{m+1-i}}=-\frac{2m+3}{2}\sum_{i=0}^{m+1}\frac{a_{2(m+1),2i}}{C_{2m+1}^{m+1-i}}-(m+1)\sum_{i=0}^{m+1}\frac{a_{2(m+1),2i}}{C_{2m+2}^{m+1-i}}. 
	\en
	Finally, using the induction hypotheses, by \eqref{2.18}, \eqref{2.19}, \eqref{2.22} and \eqref{2.23} we yield that 
	\ben
	&&\sum_{i=0}^{l+1}\frac{a_{2(l+1),2i}}{C_{2l+1}^{l+1-i}}=-\frac{2}{2l+3}\sum_{i=0}^{l+1}\frac{a_{2l+3,2i+1}}{C_{2l+2}^{l+1-i}}=\frac{(2l+2)!}{2^{2l+1}}, \\ 
	&&(l+2)\sum_{i=0}^{l+1}\frac{a_{2(l+1),2i}}{C_{2l+2}^{l+1-i}}=-\sum_{i=0}^{l+1}\frac{a_{2l+3,2i+1}}{C_{2l+3}^{l+1-i}}=\frac{(2l+4)!}{2^{2l+3}}. 
	\enn 
	In other words, \eqref{2.9} and \eqref{2.10} are satisfied when $m=l+1$. Hence, the proof is finished by induction. 
\end{proof}
\begin{lemma}\label{lem2.4}
	Follow the notations $x_0$ and $z_j$ in Theorem {\rm \ref{thm2.2}}. Suppose $x_0=(0,0,0)$ and $\nu(x_0)=(1,0,0)$. For $m\in\N$, define 
	\ben
	  &&I_{2m-1}:=\int_{\Om}\frac{1}{|y|}\left(\pa_1^{2m}\frac{1}{|y-z_j|}\right)y_1^{2m-1}dy, \\
	  &&I_{2m}:=\int_{\Om}\frac{1}{|y|}\pa_1\left[\left(\pa_1^{2m}\frac{1}{|y-z_j|}\right)y_1^{2m}\right]dy. 
	\enn
	Then for arbitrarily fixed $m\in\N$, $|I_m|\rightarrow+\infty$ as $j\rightarrow+\infty$. 
\end{lemma}
\begin{proof}
	The integral is clearly bounded uniformly for $j\in\N$ over the part that away from $x_0$. Further, since $\pa\Om\in C^\infty$, we can straighten the boundary near $x_0$. Thus, there is no loss of genearlity to assume that $\Om=\{y\in\R^3,y_2^2+y_3^2<1,-1<y_1<0\}$. For simplicity, we also set $\delta=1$ in the definition of $z_j$. Then from Lemma \ref{lem2.3}, it is deduced that 
	\ben
	  &&I_{2m-1}=\int_{\Om}\frac{1}{|y|}\frac{P_{2m}(y_1-\frac{1}{j},\sqrt{y_2^2+y_3^2})}{|y-z_j|^{4m+1}}y_1^{2m-1}dy, \\
	  &&I_{2m}=2mI_{2m-1}+\int_{\Om}\frac{1}{|y|}\frac{P_{2m+1}(y_1-\frac{1}{j},\sqrt{y_2^2+y_3^2})}{|y-z_j|^{4m+3}}y_1^{2m}dy. 
	\enn
	
	We proceed to analyze the integral $I_{2m-1}$ by the explicit expression of $\Om$ and $P_{2m}$. It follows that 
	\ben
	  I_{2m-1}\;&&=\sum_{i=0}^ma_{2m,2i}\int_{\Om}\frac{1}{|y|}\frac{(y_1-\frac{1}{j})^{2i}(y_2^2+y_3^2)^{m-i}}{|y-z_j|^{4m+1}}y_1^{2m-1}dy \\
	  &&=2\pi\sum_{i=0}^ma_{2m,2i}\int_{-1}^{0}\int_{0}^{1}\frac{1}{\sqrt{y_1^2+r^2}}\frac{(y_1-\frac{1}{j})^{2i}r^{2(m-i)}}{\sqrt{(y_1-\frac{1}{j})^2+r^2}^{4m+1}}y_1^{2m-1}rdrdy_1 \\
	  &&=-\pi\sum_{i=0}^ma_{2m,2i}\int_{0}^{1}\int_{0}^{1}\frac{1}{\sqrt{y_1^2+r}}\frac{(y_1+\frac{1}{j})^{2i}r^{m-i}}{\sqrt{(y_1+\frac{1}{j})^2+r}^{4m+1}}y_1^{2m-1}drdy_1. 
	\enn
	We spilt $I_{2m-1}$ into two parts. For $0\leq i\leq m$, define 
	\ben
	  &&I_{2m-1,i}^{(1)}:=\int_{0}^{1}\int_{0}^{1}\frac{(y_1+\frac{1}{j})^{2i}r^{m-i}}{((y_1+\frac{1}{j})^2+r)^{2m+1}}y_1^{2m-1}drdy_1, \\
	  &&I_{2m-1,i}^{(2)}:=\int_{0}^{1}\int_{0}^{1}\left(\frac{1}{\sqrt{y_1^2+r}}-\frac{1}{\sqrt{(y_1+\frac{1}{j})^2+r}}\right)\frac{(y_1+\frac{1}{j})^{2i}r^{m-i}}{\sqrt{(y_1+\frac{1}{j})^2+r}^{4m+1}}y_1^{2m-1}drdy_1. 
	\enn
	Clearly, $I_{2m-1}=-\pi\sum_{i=0}^ma_{2m,2i}(I_{2m-1,i}^{(1)}+I_{2m-1,i}^{(2)})$. From the basic knowledge of calculus, it is seen that 
	\ben
	  I_{2m-1,i}^{(1)}\;&&=\sum_{l=0}^{m-i}\int_{0}^1(y_1+\frac{1}{j})^{2(m-l)}y_1^{2m-1}\left(\int_{(y_1+\frac{1}{j})^2}^{(y_1+\frac{1}{j})^2+1}(-1)^{m-i-l}C_{m-i}^l\frac{1}{s^{2m+1-l}}ds\right)dy_1 \\
	  &&=\left(\sum_{l=0}^{m-i}\frac{(-1)^{m-i-l}C_{m-i}^l}{2m-l}\right)\int_{0}^{1}\frac{y_1^{2m-1}}{(y_1+\frac{1}{j})^{2m}}dy_1+O(1) \\
	  &&=\left(\sum_{l=0}^{m-i}\frac{(-1)^{l}C_{m-i}^l}{m+i+l}\right)\left[\int_{\frac{1}{j}}^{\frac{1}{j}+1}\sum_{l=0}^{2m-1}C_{2m-1}^l\left(-\frac{1}{j}\right)^{2m-1-l}\frac{1}{s^{2m-l}}ds\right]+O(1) \\
	  &&=\frac{1}{2mC_{2m-1}^{m-i}}\ln j+O(1), 
	\enn
	where we note that the identity 
	\ben
	  \sum_{l=0}^{m-i}\frac{(-1)^{l}C_{m-i}^l}{m+i+l}=\frac{1}{2mC_{2m-1}^{m-i}}
	\enn
	holds. 
	Further, direct calculation shows that 
	\ben
	  &&\quad\left|\frac{1}{\sqrt{y_1^2+r}}-\frac{1}{\sqrt{(y_1+\frac{1}{j})^2+r}}\right| \\
	  &&=\left|\frac{\frac{1}{j}(2y_1+\frac{1}{j})}{\sqrt{y_1^2+r}\sqrt{(y_1+\frac{1}{j})^2+r}(\sqrt{y_1^2+r}+\sqrt{(y_1+\frac{1}{j})^2+r})}\right| \\
	  &&\leq\frac{1}{j}\frac{1}{\sqrt{y_1^2+r}}\frac{1}{\sqrt{(y_1+\frac{1}{j})^2+r}} 
	\enn
	for $(y_1,r)\in(0,1)^2$, which implies 
	\ben
	  |I_{2m-1,i}^{(2)}|\;&&\leq\frac{1}{j}\int_{0}^{1}\int_{0}^{1}\frac{1}{\sqrt{y_1^2+r}}\frac{(y_1+\frac{1}{j})^{2i}r^{m-i}}{((y_1+\frac{1}{j})^2+r)^{2m+1}}y_1^{2m-1}drdy_1 \\
	  &&\leq\frac{1}{j}\int_{0}^{1}\int_{0}^{1}\frac{(y_1+\frac{1}{j})^{2i}r^{m-i}}{((y_1+\frac{1}{j})^2+r)^{2m+1}}y_1^{2m-2}drdy_1. 
	\enn
	Estimating similarily as $I_{2m-1,i}^{(1)}$, we derive that 
	\ben
	  \int_{0}^{1}\int_{0}^{1}\frac{(y_1+\frac{1}{j})^{2i}r^{m-i}}{((y_1+\frac{1}{j})^2+r)^{2m+1}}y_1^{2m-2}drdy_1\leq Cj
	\enn
	as $j\rightarrow+\infty$ with $C$ a positive constant, which indicates that $I_{2m-1,i}^{(2)}$ is bounded uniformly for $j\in\N$. Therefore, we obtain that 
	\ben
	  I_{2m-1}=-\pi\left(\frac{1}{2m}\sum_{i=0}^m\frac{a_{2m,2i}}{C_{2m-1}^{m-i}}\right)\ln j+O(1). 
	\enn
	Analogously, we can deduce that 
	\ben
	  I_{2m}=-\pi\left(\sum_{i=0}^m\frac{a_{2m,2i}}{C_{2m-1}^{m-i}}+\frac{1}{2m+1}\sum_{i=0}^m\frac{a_{2m+1,2i+1}}{C_{2m}^{m-i}}\right)\ln j+O(1). 
	\enn
	By Lemma \ref{lem2.3}, we then have 
	\ben
	  &&I_{2m-1}=-\pi\frac{(2m-1)!}{2^{2m-1}}\ln j+O(1), \\
	  &&I_{2m}=-\pi\frac{(2m)!}{2^{2m}}\ln j+O(1). 
	\enn
	Thus, the assertion follows. 
\end{proof}
\begin{theorem}\label{thm2.5}
	Follow the notations $x_0$ and $z_j$ in Theorem {\rm \ref{thm2.2}}. For a function $q_0\in C^\infty(\ov\Om,\C)$ and 
	fixed $m\in\N\cup\{-1,0\}$, if $D^\alpha q_0=0$ on $\Gamma_\varepsilon:=\pa\Om\cap B_\varepsilon(x_0)$ with $|\alpha|\leq m$ and $\varepsilon>0$ small, and $\pa_\nu^{m+1}q_0(x_0)\neq0$ {\rm (}if $m=-1$ then these simply means $q(x_0)\neq0${\rm )}, then 
	\ben
	\left\|\int_{\Om}\frac{1}{|x-y|}\frac{q_0(y)}{|y-z_j|}dy\right\|_{H^{m+7/2}(\pa\Om)}+\left\|\pa_\nu\int_{\Om}\frac{1}{|x-y|}\frac{q_0(y)}{|y-z_j|}dy\right\|_{H^{m+5/2}(\pa\Om)}\rightarrow+\infty
	\enn
	as $j\rightarrow+\infty$. 
\end{theorem}
\begin{proof}
	Without loss of genearlity, we again suppose $x_0=(0,0,0)$ and $\nu(x_0)=(1,0,0)$. We start from the case that $q_0(x_0)\neq0$. Assume in the contrary that the norms are bounded for a subsequence $\{j_l\}_{l\in\N}\subset\N$. It is known that $(q_0(x)-q_0(x_0))/|x-z_j|$ is uniformly bounded in $H^1(\Om)$ for $j\in\N$. By the regularity of volume potential (see Theorem \ref{thm2.1}) and the trace theorem, since $q_0(x_0)\neq0$, we obtain that 
	\ben
	  \left\|\int_{\Om}\frac{1}{|x-y|}\frac{1}{|y-z_{j_l}|}dy\right\|_{H^{5/2}(\pa\Om)}+\left\|\pa_\nu\int_{\Om}\frac{1}{|x-y|}\frac{1}{|y-z_{j_l}|}dy\right\|_{H^{3/2}(\pa\Om)}\leq C 
	\enn
	for $l\in\N$ with $C>0$ a constant, which implies that 
	\ben
	  \left\|\pa_1\int_{\Om}\frac{1}{|x-y|}\frac{1}{|y-z_{j_l}|}dy\right\|_{H^{3/2}(\pa\Om)}\leq C. 
	\enn
	Due to the bounded embedding from $H^{3/2}(\pa\Om)$ into $C(\pa\Om)$, we further derive that 
	\ben
	\left\|\int_{\Om}\frac{y_1-x_1}{|x-y|^3}\frac{1}{|y-z_{j_l}|}dy\right\|_{C(\pa\Om)}\leq C. 
	\enn
	 Now we claim that the above integral would blow up with $x=x_0$ and as $l\rightarrow+\infty$, which leads to a contradiction and finishes the proof when $q_0(x_0)\neq0$. To this end, for the same reason in Lemma \ref{lem2.4}, we can again suppose that $\Om=\{y\in\R^3,y_2^2+y_3^2<1,-1<y_1<0\}$. Then it follows that 
	 \ben
	   \int_{\Om}\frac{-y_1}{|y|^3}\frac{1}{|y-z_{j_l}|}dy\;&&\geq\int_{\Om}\frac{-y_1}{|y-z_{j_l}|^4}dy \\
	   &&=\pi\int_{0}^{1}\int_{0}^{1}\frac{y_1}{((y_1+\frac{1}{j_l})^2+r)^2}drdy_1 \\
	   &&=\pi\int_{0}^{1}\left(\frac{y_1}{(y_1+\frac{1}{j_l})^2}-\frac{y_1}{(y_1+\frac{1}{j_l})^2+1}\right)dy_1 \\
	   &&\geq\pi\int_{0}^{1}\left(\frac{1}{y_1+\frac{1}{j_l}}-\frac{1}{j_l}\frac{1}{(y_1+\frac{1}{j_l})^2}\right)dy_1-\pi \\
	   &&=\pi\left(\ln j_l+\ln(1+\frac{1}{j_l})+\frac{1}{j_l+1}-1\right)-\pi, 
	 \enn
	 which confirms the claim. 
	 
	 Next we deal with the remaining part still by contradiction. Suppose that the norms are bounded for $\{j_l\}_{l\in\N}$. Clearly, the norms are also bounded when restricted on $\Gamma_\varepsilon$, i.e., 
	 \ben
	 \left\|\int_{\Om}\frac{1}{|x-y|}\frac{q_0(y)}{|y-z_{j_l}|}dy\right\|_{H^{m+7/2}(\Gamma_\varepsilon)}+\left\|\pa_\nu\int_{\Om}\frac{1}{|x-y|}\frac{q_0(y)}{|y-z_{j_l}|}dy\right\|_{H^{m+5/2}(\Gamma_\varepsilon)}\leq C. 
	 \enn
	 Since $\pa\Om\in C^\infty$, for $\varepsilon>0$ sufficiently small, we can smoothly straighten the boundary $\Gamma_\varepsilon$ such that the norms after straightening are equivalent (see \cite[Theorem 3.23]{WM00}). Hence, there is no loss of genearlity to assume that $\Gamma_\varepsilon\subset\{y\in\R^3,y_1=0\}$. 
	 
	 For $m=2i_0$ even ($i_0\in\N\cup\{0\}$), it is then deduced that 
	 \ben
	   \left\|\Grad\int_{\Om}\frac{1}{|x-y|}\frac{q_0(y)}{|y-z_{j_l}|}dy\right\|_{H^{m+5/2}(\Gamma_\varepsilon)}\leq C
	 \enn
	 with $\Grad$ the surface gradient. Since $\Gamma_\varepsilon\subset\{y\in\R^3,y_1=0\}$, we have $\Grad=(\pa_2,\pa_3)$. From the condition that $q_0=0$ on $\Gamma_\varepsilon$, integration by parts yields that 
	 \ben
	   \pa_i\int_{\Om}\frac{1}{|x-y|}\frac{q_0(y)}{|y-z_{j_l}|}dy=\int_{\Om}\frac{1}{|x-y|}\pa_{i}\frac{q_0(y)}{|y-z_{j_l}|}dy-\int_{\pa\Om\setminus\ov\Gamma_\varepsilon}\frac{1}{|x-y|}\frac{q_0(y)}{|y-z_{j_l}|}\nu_i(y)ds(y)
	 \enn
	 for $x\in\ov\Om$ and $i=1,2,3$. Evidently, we have 
	 \ben
	   \left\|\int_{\pa\Om\setminus\ov\Gamma_\varepsilon}\frac{1}{|x-y|}\frac{q_0(y)}{|y-z_{j_l}|}\nu_i(y)ds(y)\right\|_{H^{m+5/2}(\Gamma_\varepsilon)}\leq C. 
	 \enn
	 By the assumptions that $\nu(x_0)=(1,0,0)$ and $D^\alpha q_0=0$ on $\Gamma_\varepsilon$ when $|\alpha|\leq m$, we see that $D^\alpha q_0(x_0)=0$ for $\alpha=(\alpha_1,\alpha_2,\alpha_3)$ with $\alpha_1\leq m$. Therefore, for $i=2,3$, it follows that 
	 \ben
	   &&\left|D^\alpha\left(\frac{\pa_{i}q_0(x)}{|x-z_{j_l}|}\right)\right|\leq C,~x\in\Om,~|\alpha|\leq m, \\
	   &&\left|D^\alpha\left(\frac{\pa_{i}q_0(x)}{|x-z_{j_l}|}\right)\right|\leq\frac{C}{|x-z_{j_l}|},~x\in\Om,~|\alpha|=m+1 
	 \enn
	 with the constant $C>0$ independent of $l\in\N$, which implies that $\pa_{i} q_0/|x-z_{j_l}|$ is uniformly bounded in $H^{m+1}(\Om)$. Thus, we obtain that 
	 \ben
	   \left\|\int_{\Om}\frac{1}{|x-y|}q_0(y)\pa_{i_1}\frac{1}{|y-z_{j_l}|}dy\right\|_{H^{m+5/2}(\Gamma_\varepsilon)}\leq C
	 \enn
	 for $i_1=2,3$ and $l\in\N$, which leads to 
	 \ben
	 \left\|\Grad\int_{\Om}\frac{1}{|x-y|}q_0(y)\pa_{i_1}\frac{1}{|y-z_{j_l}|}dy\right\|_{H^{m+3/2}(\Gamma_\varepsilon)}\leq C. 
	 \enn
	 Analyzing analogously, it is derived that 
	 \ben
	   \left\|\int_{\Om}\frac{1}{|x-y|}q_0(y)\pa_{i_1}\pa_{i_2}\frac{1}{|y-z_{j_l}|}dy\right\|_{H^{m+3/2}(\Gamma_\varepsilon)}\leq C
	 \enn
	 for $i_1,i_2=2,3$. By taking $(i_1,i_2)=(2,2),(3,3)$, since $1/|y-z_{j_l}|$ is harmonic in $\Om$, it is further deduced that 
	 \ben
	 \left\|\int_{\Om}\frac{1}{|x-y|}q_0(y)\pa_{1}^2\frac{1}{|y-z_{j_l}|}dy\right\|_{H^{m+3/2}(\Gamma_\varepsilon)}\leq C. 
	 \enn
	 Repeating the procedure $i_0$ times, finally we would obtain that 
	 \ben
	   \left\|\int_{\Om}\frac{1}{|x-y|}q_0(y)\pa_{1}^{2i_0+2}\frac{1}{|y-z_{j_l}|}dy\right\|_{H^{3/2}(\Gamma_\varepsilon)}\leq C, 
	 \enn
	 which indicates that 
	 \ben
	   \left|\int_{\Om}\frac{1}{|y|}q_0(y)\pa_{1}^{2i_0+2}\frac{1}{|y-z_{j_l}|}dy\right|\leq C 
	 \enn
	 uniformly for $l\in\N$ due to the bounded embedding from $H^{3/2}(\Gamma_\varepsilon)$ into $C(\Gamma_\varepsilon)$. We further yield by the Taylor's expansion that 
	 \ben
	   \left|\left(q_0(y)-\frac{\pa_1^{2i_0+1}q_0(x_0)}{(2i_0+1)!}y_1^{2i_0+1}\right)\pa_{1}^{2i_0+2}\frac{1}{|y-z_{j_l}|}\right|\leq\frac{C}{|y-z_{j_l}|},~y\in\Om. 
	 \enn
	 Since $\pa_{1}^{m+1}q_0(x_0)=\pa_{1}^{2i_0+1}q_0(x_0)\neq0$, we obtain that 
	 \ben
	   \left|\int_{\Om}\frac{1}{|y|}y_1^{2i_0+1}\pa_{1}^{2i_0+2}\frac{1}{|y-z_{j_l}|}dy\right|\leq C, 
	 \enn
	 which clearly contradicts with Lemma \ref{lem2.4}. Therefore, we prove the case when $m$ is even. 
	 
	 As for the case that $m=2i_0+1$ is odd ($i_0\in\N\cup\{0\}$), from the estimates that 
	 \ben
	 \left\|\Grad\pa_\nu\int_{\Om}\frac{1}{|x-y|}\frac{q_0(y)}{|y-z_{j_l}|}dy\right\|_{H^{m+3/2}(\Gamma_\varepsilon)}\leq C, 
	 \enn
	 utilizing the same methods, we can derive that 
	 \ben
	   \left|\int_{\Om}\frac{1}{|y|}\pa_1\left[\left(\pa_1^{2i_0+2}\frac{1}{|y-z_{j_l}|}\right)y_1^{2i_0+2}\right]dy\right|\leq C, 
	 \enn
	 which contradicts with Lemma \ref{lem2.4} again. The proof is thus finished. 
\end{proof}

\section{Proofs of main results}\label{sec3}
\setcounter{equation}{0}
In this section, we first study some further properties of $\psi_{m,j}$ defined in Theorem \ref{thm2.2}. Then we give the proof of the main results Theorems \ref{thm1.1} and \ref{thm1.2}. 

Before going further, we introduce some notations. For $(q_i,D_i,\mathcal{B}_i)$, $i=1,2$, denote by $u_j^{(i)},\varphi_{m,j}^{(i)},\psi_{m,j}^{(i)}$ the functions defined in Theorem \ref{thm2.2} with $m\in\N\cup\{-1,0\}$ and $j\in\N$. 
\begin{theorem}\label{thm3.1}
	For any fixed $x_0\in\pa\Om$ and $m_0\in\N\cup\{0\}$, suppose $D^\alpha q_1=D^\alpha q_2$ on $\Gamma_\varepsilon$ when $|\alpha|\leq m_0$. Define the corresponding $z_j,u_j^{(i)},\varphi_{m,j}^{(i)},\psi_{m,j}^{(i)}$. Then it follows that 
	\be\label{3.1}
	  &&\|\psi_{m_0+1,j}^{(1)}-\psi_{m_0+1,j}^{(2)}\|_{H^{m_0+3}(\Om)}\leq C, \\ \label{3.2}
	  &&\|D^\alpha q_1\psi_{m_0+1,j}^{(1)}-D^\alpha q_2\psi_{m_0+1,j}^{(2)}\|_{H^{m_0+3-|\alpha|}(\Om)}\leq C,
	\en
	with $|\alpha|\leq m_0+2$, where the positive constant $C$ is independent of $j\in\N$. 
\end{theorem}
\begin{proof}
	We first consider the case $m_0=0$. From the definition of $\psi_{1,j}^{(i)}$, $i=1,2$, it can be verified that 
	\ben
	\left\{
	\begin{array}{ll}
		\Delta(\psi_{1,j}^{(1)}-\psi_{1,j}^{(2)})=-(q_1-q_2)\Phi(\cdot,z_j)~~~&{\rm in}~\Om,\\
	    \psi_{1,j}^{(1)}-\psi_{1,j}^{(2)}=\psi_{1,j}^{(1)}-\psi_{1,j}^{(2)}~~~&{\rm on}~\pa\Om. 
	\end{array}
	\right.
	\enn
	Since $q_1=q_2$ on $\Gamma_\varepsilon$, we know that $(q_1-q_2)\Phi(\cdot,z_j)$ are uniformly bounded in $H^1(\Om)$ for $j\in\N$. By Remark \ref{remark2.3}, we also have $\|\psi_{1,j}^{(1)}-\psi_{1,j}^{(2)}\|_{H^{5/2}(\pa\Om)}\leq C$. Therefore, from the regularity of Poisson's equation, it follows that $\|\psi_{1,j}^{(1)}-\psi_{1,j}^{(2)}\|_{H^{3}(\Om)}\leq C$. Further, it is easy to see that  
	\be\label{3.3}
	  D^\alpha q_1\psi_{1,j}^{(1)}-D^\alpha q_2\psi_{1,j}^{(2)}=D^\alpha(q_1-q_2)\psi_{1,j}^{(1)}+D^\alpha q_2(\psi_{1,j}^{(1)}-\psi_{1,j}^{(2)}). 
	\en
	Since $q_1,q_2\in C^\infty(\ov\Om,\C)$, it then suffices to show that 
	\be\label{3.4}
	  \|D^\alpha(q_1-q_2)\psi_{1,j}^{(1)}\|_{H^{3-|\alpha|}(\Om)}\leq C 
	\en
	for $|\alpha|\leq2$. Again by Remark \ref{remark2.3}, we have that $\|\psi_{1,j}^{(1)}\|_{H^2(\Om)}\leq C$, which clearly implies the estimates \eqref{3.4} holds when $|\alpha|=1,2$. Direct calculation yields that 
	\ben
		\Delta[(q_1-q_2)\psi_{1,j}^{(1)}]=\Delta(q_1-q_2)\psi_{1,j}^{(1)}+2\na(q_1-q_2)\cdot\na\psi_{1,j}^{(1)}-(q_1-q_2)(q_1\Phi(\cdot,z_j))
	\enn
	in $\Om$. Since $\|(q_1-q_2)\psi_{1,j}^{(1)}\|_{H^{5/2}(\pa\Om)}\leq C$, we obtain that  \eqref{3.4} with $|\alpha|=0$ holds. 
	
	Now we finish the proof by induction. Suppose \eqref{3.1} and \eqref{3.2} hold when $m_0\leq l$ with $l\geq0$. We shall prove these estimates when $m_0=l+1$. It is deduced from the definition of $\psi_{l+2,j}^{(i)}$, $i=1,2$ that 
	\ben
	\left\{
	\begin{array}{ll}
		\Delta(\psi_{l+2,j}^{(1)}-\psi_{l+2,j}^{(2)})=-(q_1\psi_{l,j}^{(1)}-q_2\psi_{l,j}^{(2)})-(q_1-q_2)\Phi(\cdot,z_j)~~~&{\rm in}~\Om,\\
		\psi_{l+2,j}^{(1)}-\psi_{l+2,j}^{(2)}=\psi_{l+2,j}^{(1)}-\psi_{l+2,j}^{(2)}~~~&{\rm on}~\pa\Om. 
	\end{array}
	\right.
	\enn
	Since $D^\alpha q_1=D^\alpha q_2$ on $\Gamma_\varepsilon$ when $|\alpha|\leq l+1$, it follows from the Taylor's expansion that for $x\in\Om$ 
	\ben
	  &&D^\alpha[(q_1-q_2)\Phi(x,z_j)]\leq C,~~~|\alpha|\leq l+1, \\
	  &&D^\alpha[(q_1-q_2)\Phi(x,z_j)]\leq\frac{C}{|x-z_j|},~~~|\alpha|=l+2, 
	\enn
	which implies that $(q_1-q_2)\Phi(\cdot,z_j)$ is uniformly bounded in $H^{l+2}(\Om)$ for $j\in\N$. By Remark \ref{remark2.3}, we see $\|\psi_{l+2,j}^{(1)}-\psi_{l+2,j}^{(2)}\|_{H^{l+7/2}(\pa\Om)}\leq C$. If $l=0$, clearly we have $\|q_1\psi_{l,j}^{(1)}-q_2\psi_{l,j}^{(2)}\|_{H^{l+2}(\Om)}\leq C$. If $l>0$, from the induction hypothesis, we know the estimate \eqref{3.2} with $m_0=l-1$ and $|\alpha|=0$ holds, which is $\|q_1\psi_{l,j}^{(1)}-q_2\psi_{l,j}^{(2)}\|_{H^{l+2}(\Om)}\leq C$. Therefore, the regularity of Poisson's equation gives $\|\psi_{l+2,j}^{(1)}-\psi_{l+2,j}^{(2)}\|_{H^{l+4}(\Om)}\leq C$. In other words, \eqref{3.1} holds when $m_0=l+1$. In view of \eqref{3.3}, it is enough to prove that 
	\be\label{3.5}
	  \|D^\alpha(q_1-q_2)\psi_{l+2,j}^{(1)}\|_{H^{l+4-|\alpha|}(\Om)}\leq C 
	\en
	for $|\alpha|\leq l+3$. Since $\|\psi_{l+2,j}^{(1)}\|_{H^2(\Om)}\leq C$, again we see that $\eqref{3.5}$ holds for $|\alpha|=l+2,l+3$. Next we show that if \eqref{3.5} holds for $l_0\leq|\alpha|\leq l+3$ with $0<l_0\leq l+2$, then \eqref{3.5} is satisfied when $|\alpha|=l_0-1$, which clearly completes the proof. For $|\gamma|=l_0-1\leq l+1$, it is derived that 
	\ben
	  \Delta[D^\gamma(q_1-q_2)\psi_{l+2,j}^{(1)}]=\;&&\Delta D^\gamma(q_1-q_2)\psi_{l+2,j}^{(1)}+2\na D^\gamma(q_1-q_2)\cdot\na\psi_{l+2,j}^{(1)} \\
	  &&-D^\gamma(q_1-q_2)(q_1\Phi(\cdot,z_j)+\psi_{l,j}^{(1)}) 
	\enn
	in $\Om$ and $\|D^\gamma(q_1-q_2)\psi_{l+2,j}^{(1)}\|_{H^{l+7/2}(\pa\Om)}\leq C$ by Remark \ref{remark2.3}. Utilizing the estimates \eqref{3.5} with $|\alpha|=l_0+1$, we deduce that $\|\Delta D^\gamma(q_1-q_2)\psi_{l+2,j}^{(1)}\|_{H^{l+3-l_0}(\Om)}\leq C$. From the observation  
	\ben
	  \pa_iD^\gamma(q_1-q_2)\pa_i\psi_{l+2,j}^{(1)}=\pa_i(\pa_iD^\gamma(q_1-q_2)\psi_{l+2,j}^{(1)})-\pa_i^2D^\gamma(q_1-q_2)\psi_{l+2,j}^{(1)}
	\enn
	with $i=1,2,3$ and the estimates \eqref{3.5} with $|\alpha|=l_0,l_0+1$, it further follows that $\|\na D^\gamma(q_1-q_2)\cdot\na\psi_{l+2,j}^{(1)}\|_{H^{l+3-l_0}(\Om)}\leq C$. Since $D^\beta[D^\gamma(q_1-q_2)]=0$ on $\Gamma_\varepsilon$ with $|\beta|\leq l+2-l_0$, we know that $D^\gamma(q_1-q_2)q_1\Phi(\cdot,z_j)$ is uniformly bounded in $H^{l+3-l_0}(\Om)$ for $j\in\N$. Moreover, the estimate \eqref{3.2} with $m_0=l-1$ and $|\alpha|=l_0-1$ tells that $\|D^\gamma(q_1-q_2)\psi_{l,j}^{(1)}\|_{H^{l+3-l_0}(\Om)}\leq C$. Thus, by the regularity of Poisson's equation, we conclude that \eqref{3.5} holds when $|\alpha|=l_0-1$ and the proof is finally complete. 
\end{proof}

With all the preceding preparations, we now are at the position to prove the main results of this paper. 
\begin{proof}{(of Theorem \ref{thm1.1})}
	We prove by contradiction. Assume there exists a $l\in\N\cup\{-1,0\}$ and a $x_0\in\Gamma$ such that $D^\alpha q_1=D^\alpha q_2$ on $\Gamma$ when $|\alpha|\leq l$ and $\pa_\nu^{l+1}q_1(x_0)\neq\pa_\nu^{l+1}q_2(x_0)$. Define the corresponding $z_j,u_j^{(i)},\varphi_{m,j}^{(i)},\psi_{m,j}^{(i)}$. We can further find a small $\varepsilon>0$ such that $\pa\Om\cap B_{2\varepsilon}(x_0)\subset\Gamma$. Let $\eta\in C^\infty(\R^3)$ be the cut-off function satisfying that $\eta=1$ in $B_\varepsilon(x_0)$ and $\eta=0$ in $\R^3\setminus\ov{B_{2\varepsilon}(x_0)}$. Define $\hat u_j^{(i)}$, $i=1,2$, to be the unique solution of problem \eqref{1.1} with the boundary data $f_j=\eta\Phi(\cdot,z_j)$, $j\in\N$. Clearly, for fixed $m\in\N$ and $i=1,2$, $\|\hat u_j^{(i)}-u_j^{(i)}\|_{H^m(\Om\setminus\ov D)}\leq C$ uniformly for $j\in\N$. Since $\Lambda_1=\Lambda_2$ locally on $\Gamma$, we have $\pa_\nu\hat u_j^{(1)}=\pa_\nu\hat u_j^{(2)}$ on $\Gamma$ with $j\in\N$. Then by Theorem \ref{thm2.2} and Remark \ref{remark2.3}, we deduce that for $m\in\N$ 
	\ben
	\|\varphi_{m,j}^{(1)}-\varphi_{m,j}^{(2)}\|_{H^{m+1/2}(\pa\Om)}\leq C~~{\rm and}~~\|\psi_{m,j}^{(1)}-\psi_{m,j}^{(2)}\|_{H^{m+3/2}(\pa\Om)}\leq C 
	\enn
	uniformly for $j\in\N$. Taking $m=l+2$, by the definition of $\varphi_{m,j}^{(i)},\psi_{m,j}^{(i)}$ we derive that 
	\ben
	  &&\|\pa_\nu(G_{q_1}-G_{q_2})(\Phi(\cdot,z_j))+\pa_\nu(G_{q_1}\psi_{l,j}^{(1)}-G_{q_2}\psi_{l,j}^{(2)})\|_{H^{l+5/2}}(\pa\Om)\leq C\\
	  &&\|(G_{q_1}-G_{q_2})(\Phi(\cdot,z_j))+(G_{q_1}\psi_{l,j}^{(1)}-G_{q_2}\psi_{l,j}^{(2)})\|_{H^{l+7/2}}(\pa\Om)\leq C. 
	\enn
	By Theorem \ref{thm3.1}, we see from \eqref{3.2} with $m_0=l-1$ and $|\alpha|=0$ that $\|q_1\psi_{l,j}^{(1)}-q_2\psi_{l,j}^{(2)}\|_{H^{l+2}(\Om)}\leq C$. Due to Theorem \ref{2.1} and the trace theorem, we further yield that 
	\ben
	  &&\|\pa_\nu(G_{q_1}\psi_{l,j}^{(1)}-G_{q_2}\psi_{l,j}^{(2)})\|_{H^{l+5/2}}(\pa\Om)\leq C, \\
	  &&\|(G_{q_1}\psi_{l,j}^{(1)}-G_{q_2}\psi_{l,j}^{(2)})\|_{H^{l+7/2}}(\pa\Om)\leq C, 
	\enn
	which implies that 
	\ben
	\left\|\int_{\Om}\frac{1}{|x-y|}\frac{(q_1-q_2)(y)}{|y-z_j|}dy\right\|_{H^{l+7/2}(\pa\Om)}+\left\|\pa_\nu\int_{\Om}\frac{1}{|x-y|}\frac{(q_1-q_2)(y)}{|y-z_j|}dy\right\|_{H^{l+5/2}(\pa\Om)}\leq C
	\enn
	uniformly for $j\in\N$. This contradicts with Theorem \ref{thm2.5} and the proof is thus finished. 
\end{proof}
\begin{proof}{(of Theorem \ref{thm1.2})}
	By Theorem \ref{thm1.1}, since $q_i$ is analytic we immediately obtain that $q_1=q_2$ in $U:=\Om\setminus(\ov D_1\cup\ov D_2)$. It then suffices to show $D_1=D_2$ and $\mathcal{B}_1=\mathcal{B}_2$. Now we introduce the following Green's function $G^{(i)}(x,y)$, $i=1,2$, 
	\be\label{3.7}
	\left\{
	\begin{array}{ll}
		\Delta G^{(i)}(x,y)+q_iG^{(i)}(x,y)=\delta_y~~~&{\rm in}~\Om\setminus\ov D_i,\\
		G^{(i)}(x,y)=0~~~&{\rm on}~\pa\Om, \\
		\mathcal{B}_i(G^{(i)}(x,y))=0~~~&{\rm on}~\pa D_i. 
	\end{array}
	\right.
	\en
	It is well known that problem \eqref{3.7} is well-posed with $y\in\Om\setminus\ov D_i$ and $G^{(i)}(x,y)=G^{(i)}(y,x)$ for $x,y\in\Om\setminus\ov D_i$, $x\neq y$. Further, due to the fact that $\Lambda_1=\Lambda_2$ locally on $\Gamma$ and $q_1=q_2$ in $U$, following the same lines as in \cite{JMK21} (see details in the arguments between (3.44) and (3.47) in \cite{JMK21}), it is derived that $G^{(1)}(x,y)=G^{(2)}(x,y)$ with $x,y\in\ov U$, $x\neq y$. Define $\Psi^{(i)}(x,y):=G^{(i)}(x,y)-\Phi(x,y)$, $i=1,2$. It can be verified that 
	\be\label{3.8}
	\left\{
	\begin{array}{ll}
		\Delta \Psi^{(i)}(x,y)+q_i\Psi^{(i)}(x,y)=-q_i\Phi(x,y)~~~&{\rm in}~\Om\setminus\ov D_i,\\
		\Psi^{(i)}(x,y)=-\Phi(x,y)~~~&{\rm on}~\pa\Om, \\ 
		\mathcal{B}_i(\Psi^{(i)}(x,y))=-\mathcal{B}_i(\Phi(x,y))~~~&{\rm on}~\pa D_i. 
	\end{array}
	\right.
	\en
	
	We first prove that $D_1=D_2=D$. Assuming the oppisite, then there exists a $y_0\in(\pa U\setminus\pa\Om)\cap\pa D_1$ and a $\delta>0$ sufficiently small such that $B_{2\delta}(y_0)\cap\ov D_2=B_{2\delta}(y_0)\cap\pa\Om=\emptyset$. Define 
	\ben
	  y_j:=y_0+\frac{\delta}{j}\nu(y_0)\in B_{2\delta}(y_0)\cap U,~j\in\N 
	\enn
	with $\nu(y_0)$ the unit exterior normal on $y_0\in\pa D_1$. We have $G^{(1)}(x,y_j)=G^{(2)}(x,y_j)$ and thus $\Psi^{(1)}(x,y_j)=\Psi^{(2)}(x,y_j)$ for $x\in\ov U$ and $j\in\N$. Since $\{y_j\}$ has a positive distance from $\pa D_2$ and $\pa\Om$, by \eqref{3.8} and the elliptic regularity we obtain that $\|\Psi^{(2)}(x,y_j)\|_{H^1(\Om\setminus\ov D_2)}\leq C$ uniformly for $j\in\N$. If $\mathcal{B}_1(u)=u$, by the trace theorem we see that $\|\Phi(x,y_j)\|_{H^{1/2}(B_{2\delta}(y_0)\cap\pa D_1)}=\|\Psi^{(1)}(x,y_j)\|_{H^{1/2}(B_{2\delta}(y_0)\cap\pa D_1)}=\|\Psi^{(2)}(x,y_j)\|_{H^{1/2}(B_{2\delta}(y_0)\cap\pa D_1)}\leq C$, which is a contradiction. If $\mathcal{B}_1(u)=\pa_\nu u+\gamma_1u$, similarly, it also leads to a contradiction. Therefore, $D_1=D_2=D$. 
	
	Now we show $\mathcal{B}_1=\mathcal{B}_2$. Suppose not, then there are two cases: (i) $\mathcal{B}_1(u)=\pa_\nu u+\gamma_1u$ and $\mathcal{B}_2(u)=\pa_\nu u+\gamma_2u$ with $\gamma_1\neq\gamma_2$, (ii) $\mathcal{B}_1(u)=u$ and $\mathcal{B}_2(u)=\pa_\nu u+\gamma_2u$. Arbitrarily choosing $y_0\in\Om\setminus\ov D$, we have that $G^{(1)}(x,y_0)=G^{(2)}(x,y_0)$ for $x\in\ov\Om\setminus D$. In case (i), from the boundary conditions in \eqref{3.7}, it is seen that 
	\ben
	  \pa_\nu G^{(1)}(x,y_0)+\gamma_1G^{(1)}(x,y_0)=\pa_\nu G^{(1)}(x,y_0)+\gamma_2G^{(1)}(x,y_0)=0~~{\rm on}~\pa D, 
	\enn
	which implies that $(\gamma_1-\gamma_2)G^{(1)}(x,y_0)=0$ on $\pa D$. Since $\gamma_1\neq\gamma_2$, by the unique continuation we deduce that $G^{(1)}(x,y_0)=0$ for $x\in\Om\setminus(\ov D\cup\{y_0\})$, which is clearly a contradiction. Dealing with case (ii) analogously, it still yields a contradiction. Thus, $\mathcal{B}_1=\mathcal{B}_2$. The proof is complete. 
\end{proof}
\begin{remark}\label{rem3.2}
	Since in the proof we only use the solutions with boundary data $\Phi(x,y)$, Theorem {\rm \ref{thm1.1}} still holds if we weaken the condition from $\Lambda_1=\Lambda_2$ locally to $(\Lambda_1-\Lambda_2)(\Phi(\cdot,y))=0$ locally  with $y\in\R^3\setminus\ov \Om$ near $\Gamma$. 
\end{remark}

\section*{Acknowledgements}
This work was supported by the NNSF of China with grant 12122114.

\end{document}